\documentclass[12pt,fleqn]{amsart}
\usepackage[all]{xy}
\usepackage{tikz-cd}
\usepackage{amscd}
\usepackage{color}
\usepackage{enumerate}
\usepackage{amssymb}
\newtheorem{theorem}{Theorem}[section]
\newtheorem{lemma}[theorem]{Lemma}
\newtheorem{proposition}[theorem]{Proposition}
\newtheorem{problem}[theorem]{Problem}

\theoremstyle{definition}
\newtheorem{definition}[theorem]{Definition}

\theoremstyle{construction}

\theoremstyle{remark}

\numberwithin{equation}{section}

\newcommand{\Ext}{\operatorname{Ext}}

\newcommand{\nnorm}[1]{{\left\vert\kern-0.25ex\left\vert\kern-0.25ex\left\vert #1
		\right\vert\kern-0.25ex\right\vert\kern-0.25ex\right\vert}}
		
\newcommand{\bA}{\mathbb A}
\newcommand{\cA}{\mathcal A}
\newcommand{\cB}{\mathcal B}
\newcommand{\cC}{\mathcal C}
\newcommand{\cP}{\mathcal P}

\newcommand{\cN}{\mathcal N}

\newcommand{\K}{\mathcal K}

\newcommand{\fB}{\mathfrak B}
\newcommand{\fC}{\mathfrak C}

\newcommand{\ult}{\protect{ult}}
\newcommand{\clop}{\protect{\rm clop}}
\newcommand{\la}{\langle}
\newcommand{\ra}{\rangle}
\newcommand{\sub}{\subseteq}
\newcommand{\eps}{\varepsilon}

\newcommand{\sm}{\setminus}
\newcommand{\con}{\mathfrak c}

\newcommand{\wt}{\widetilde}
\newcommand{\vf}{\varphi}
\renewcommand{\Bar}{\operatorname{Bar}}
\newcommand{\supp}{\operatorname{supp}}

\newcommand{\pe}{\mathfrak p}

\newcommand{\ps}{\mathcal P}
\newcommand{\ml}{M^{\ell_1}}

\begin{document}

%\showlabels{cite}

\baselineskip=17pt

 \title{On the three space property for $C(K)$-spaces}
% \title{Properties of twisted sums of $C(K)$-spaces}

\author[G.\ Plebanek]{Grzegorz Plebanek}
\address{Instytut Matematyczny\\ Uniwersytet Wroc\l awski\\ Pl.\ Grunwaldzki 2/4\\
50-384 Wroc\-\l aw\\ Poland} \email{grzes@math.uni.wroc.pl}

\author[A. Salguero Alarc\'on]{Alberto Salguero Alarc\'on}
\address{Instituto de Matem\'aticas\\ Universidad de Extremadura\\
Avenida de Elvas\\ 06071-Badajoz\\ Spain} \email{salgueroalarcon@unex.es}

\thanks{
The first author has been supported by the grant 2018/29/B/ST1/00223 from National Science Centre, Poland.
The second author benefits from a grant associated to Project IB16056 (Junta de Extremadura, Spain) and from the grant FPU18/990 (Ministerio de Ciencia, Innovaci\'on y Universidades, Spain).}


\begin{abstract}
Assuming $\pe=\con$, we show that for every Eberlein compact space $L$ of weight $\con$ there exists a short exact sequence $0\to c_0\to X\to C(L)\to 0$, where the Banach space $X$ is not isomorphic to a $C(K)$-space.
\end{abstract}

\date{}
\maketitle

%{\bf Later \ldots}
%A few words on Benyamini construction from 1977\ldots
%Intro on exact sequences\ldots References to previous papers, including our recent preprint and previous paper of Castillo, Aviles, Marciszewski Plebanek etc.
\section{Introduction}
The reader is referred to the Preliminaries section for all unexplained terminology.
This paper focuses on solving  a problem that Castillo has repeatedly posed \cite{efe,v,Ca16,2132}: must every twisted sum of $c_0(\aleph)$ and $c_0(\aleph')$  be isomorphic to a $C(K)$-space? The problem stems from an idea of Cabello  mentioned in \cite{3space} --- that a twisted sum of two $C(K)$-spaces does not need to be a $C(K)$-space itself. The proof was sketched in \cite{3space} and, in all its details, in \cite{2132}.
The argument  uses a clever construction of Benyamini \cite{ben}, a renorming $\|\cdot\|_n$ of $\ell_\infty$ such that $(\ell_\infty, \|\cdot\|_n)$ cannot be
$n$-complemented in any $C(K)$-space. This eventually results in an exact sequence

\begin{center}
 \begin{tikzcd}
	0 \arrow[r] & c_0 \arrow[r] & X \arrow[r]  & c_0(\ell_\infty/c_0) \arrow[r] & 0
\end{tikzcd}
\end{center}
in which $X$ cannot be a complemented subspace of a $C(K)$-space.

Benyamini's construction uses in an essential way properties of the compact space $\beta\omega\setminus \omega$. Its complexity suggests that such construction cannot be carried out with simpler compacta; for instance:
\begin{problem}\label{problem} Given any cardinal number $\aleph$, is every twisted sum space $X$ in
\begin{center} \begin{tikzcd}
	0 \arrow[r] & c_0 \arrow[r] & X \arrow[r]  & c_0(\aleph) \arrow[r] & 0,
\end{tikzcd} \end{center}
isomorphic to a $C(K)$-space?\end{problem}

We present here a relatively consistent negative solution to Problem \ref{problem}:
Assuming $\pe=\con$, a weak version of Martin's axiom, we prove that every Eberlein compact  space $K$ of weight $\con$  admits an exact sequence $ 0 \to c_0 \to   X \to C(K) \to 0$ where $X$ is not isomorphic to a $C(K)$-space. Recall that for any $\aleph$, the space $c_0(\aleph)$ is isomorphic to $C(\bA(\aleph))$, the space of continuous
functions on the Eberlein compact space $\bA(\aleph)$, the one-point compactification of a discrete set of size $\aleph$. Hence, our main result provides the negative answer to
Problem \ref{problem} for $\aleph=\con$.

\section{Preliminaries}
\subsection{Twisted sums of Banach spaces}
\par A \emph{short exact sequence} of Banach spaces is a diagram of Banach spaces and (continous linear) operators
\begin{center}\begin{tikzcd}
	0 \arrow[r] & Y \arrow[r, "j"] & X \arrow[r, "p"]  & Z \arrow[r] & 0
	\end{tikzcd} \end{center}
where the kernel of each arrow coincides with the image of the preceding one. We also say that $X$ is a \emph{twisted sum} of $Y$ and $Z$. In this setting, the open mapping theorem yields that $Y$ is isomorphic to a subspace of $X$ and $Z$ to the quotient $X/Y$. Two exact sequences $0\to Y \to X_i \to Z\to 0$, $i=1,2$, are \emph{equivalent} if there is an operator $u:X_1 \to X_2$ making commutative the following diagram
\begin{center}\begin{tikzcd}
0 \arrow[r] & Y \arrow[r] \arrow[d,equal] & X_1 \arrow[r] \arrow[d, "u"] & Z \arrow[r] \arrow[d,equal] & 0 \\
0 \arrow[r] & Y \arrow[r]           & X_2 \arrow[r]                & Z \arrow[r]           & 0
\end{tikzcd}\end{center}
Again, the open mapping theorem makes $u$ an isomorphism.
We say an exact sequence is \emph{trivial}, or that it splits, if the operator $p$ has a right-inverse $r:Z\to X$. In other words, when it is equivalent to the twisted sum $0\to Y \to Y\oplus Z\to Z\to 0$. When this happens, $X$ must be of course isomorphic to $Y\oplus Z$, but the converse is not true. Let us write $\Ext(Z,Y)$ to represent the space of twisted sums of $Y$ and $Z$, modulo equivalence. Hence $\Ext(Z,Y)=0$ means that the only twisted sum of $Y$ and $Z$ is the trivial one.

\par Two basic constructions regarding exact sequences are the \emph{pull-back} and the \emph{push-out}. Given a short exact sequence $0\to Y \overset i \to X \overset p \to  Z\to 0$ and an operator $T: Z'\to Z$, then there is a unique space $PB$, together with an operator $PB\to X$, which makes commutative the following diagram:
\begin{center}
\begin{tikzcd}
0 \arrow[r] & Y \arrow[r, "i"]      & X \arrow[r, "p"]       & Z \arrow[r]                 & 0 \\
0 \arrow[r] & Y \arrow[r] \arrow[u,equal] & PB \arrow[r] \arrow[u] & Z' \arrow[r] \arrow[u, "T"] & 0
\end{tikzcd}
\end{center}
Explicitly, $PB=\{(x,z')\in X\oplus Z: p(x)=T(z')\}$, with the subspace norm of $X\oplus Z$.
\par Dually, if we take the previous exact sequence and an operator $S:Y\to Y'$, there is a unique space $PO$ and an operator $X\to PO$ such that the following diagram is commutative:
\begin{center}
\begin{tikzcd}
0 \arrow[r] & Y \arrow[r, "i"] \arrow[d, "S"] & X \arrow[r, "p"] \arrow[d] & Z \arrow[r] \arrow[d,equal] & 0 \\
0 \arrow[r] & Y' \arrow[r]                    & PO \arrow[r]               & Z \arrow[r]           & 0
\end{tikzcd}
\end{center}
Explicitely, $PO=(Y'\oplus X)/\overline\Delta$ where $\Delta=\{(i(y), -S(y)): y\in Y\}$, with the quotient norm.

\subsection{$\pe=\con$}\label{pe}
Recall that \emph{Martin's axiom for $\sigma$-centered posets $\pe=\con$}  amounts to saying that whenever $\cB\sub {\mathcal P}(\omega)$ is a family  with $|\cB|<\con$ such that $B_1\cap\ldots \cap B_n$ is infinite for every $n$ and any $B_i\in\cB$, then there is an infinite set $A$ such that $A\sub^\ast B$ for every $B\in\cB$, see Fremlin \cite{Fr84}. Recall also that, in the topological setting, under  $\pe=\con$, every compact space $K$ of weight $<\con$ is sequentially compact, see \cite[24A]{Fr84}.

\subsection{$\cC$-spaces}

The symbol $K$ always stands for a compact Hausdorff space. We say $X$ is \emph{a $\cC$-space} if it is \emph{isomorphic} to a space of continuous real-valued functions $C(K)$ for some $K$. For a given space $C(K)$, we, as usual, identify the dual space $C(K)^*$ with the space $M(K)$ of signed Radon measures on $K$ of finite variation. $M_1(K)$ stands for the unit ball of $M(K)$, equipped with the $weak^\ast$ topology inherited from $C(K)^*$. For a measure $\mu\in M(K)$, $|\mu|$ denotes its variation.

%\begin{definition} Let $X$ be a Banach space $X$ and $0<c\leq 1$. We say $F \subseteq B_{X^*}$ is \emph{$c$-norming} if the composition $T= R\iota$
%$$\begin{CD} X@>\iota>> C(B_{X^*}) @>R>> C(F)\end{CD}$$
%of the canonical embedding $\iota$ with the natural restriction operator $R$ is injective, has closed range, and has an inverse $T^{-1}: T(X) \to X$ such that $\|T^{-1}\| < 1/c$. We say that $F$ is \emph{free} if $T$ is onto.\end{definition}

\begin{definition} Let $X$ be a Banach space. Given any $weak^\ast$ compact $F$ in $B_{X^*}$, we can consider the norm-one operator
	$$T: X \longrightarrow C(F) \quad , \quad x \mapsto x(f) = f(x)$$ 
	Given $0<c\leq 1$, we say $F$ is \emph{$c$-norming} if $T$ is injective, has closed range, and has an inverse $T^{-1}: T(X) \to X$ such that $\|T^{-1}\| < 1/c$.
	We say that $F$ is \emph{free} if $T$ is onto. 
\end{definition}

\begin{lemma}\label{p:1}
For a Banach space $X$, the following are equivalent:
\begin{enumerate}[(i)]
\item $X$ is  a $\cC$-space;
\item there is a free subset  $F$ of $B_{X^\ast}$ which is  $c$-norming for some $0<c\leq1$.
\end{enumerate}
\end{lemma}

\begin{proof} Assume $T:X\to C(K)$ is an isomorphism such that
$c \cdot \|x\|\le \|Tx\|\le \|x\|$ for every $x\in X$. This means that the norm $\nnorm{x} = \sup_{k\in K} |Tx(k)|$ is equivalent to the norm in $X$. Hence write $\Delta_K=\{\delta_k:k\in  K\}$ and consider $F=T^\ast(\Delta_K)$. Such $F$ is $weak^\ast$ compact, since $T^\ast$ is $weak^\ast$ continuous, and it is clearly $c$-norming and free. The converse is obvious.
\end{proof}

The main idea of our construction is to eliminate every possibility of $B_{X^*}$ having a free $c$-norming subset, so that the resulting space $X$ cannot be a $\cC$-space. The next simple lemma will enable us to recognize sets that are not free.

\begin{lemma}\label{p:2}
Let $F\sub B_{X^\ast}$ be $weak^\ast$ compact and $c$-norming for some $c>0$.
If there are distinct $x_0^\ast, x_1^\ast, x_2^\ast\in F$ such that
\[ \left\| x_0^\ast-1/2(x_1^\ast+x_2^\ast)\right\|<c,\]
then $F$ is not free.
\end{lemma}

\begin{proof}
Suppose otherwise; consider a norm-one $g\in C(F)$ such that $g(x_0^\ast)=0$ while $g(x_1^\ast)=1=g(x_2^\ast)$.
By freeness, there is $x\in X$ such that $x|F=g|F$. Note that $c\|x\|\le \|g\|=1$ since $F$ is $c$-norming, and
\[c>\left\| x_0^\ast-1/2(x_1^\ast+x_2^\ast)\right\|\ge c |x_0^\ast(x)-1/2(x_1^\ast(x)+x_2^\ast(x)|=c,\]
a contradiction.
\end{proof}

\begin{definition}\label{p:3}
In the setting of Lemma \ref{p:2}, we call $x_0^\ast, x_1^\ast, x_2^\ast\in F$
 a {\em forbidden $c$-triple}.
\end{definition}

\subsection{Almost disjoint families and Alexandroff-Urysohn compacta} \label{s:adf}
A topological space $X$ is \emph{scattered} if no nonempty subset $A\subseteq X$ is dense-in-itself.
For an ordinal $\alpha$, $X^{(\alpha)}$ denotes the $\alpha$th
Cantor-Bendixon derivative of the space $X$. For a scattered
space $X$, the scattered height is
${ht(X)}=\min\{\alpha: X^{(\alpha)} =
\emptyset\}$.

We write $\bA(\kappa)$ for the (Alexandroff) one-point compactification
of a discrete space of cardinality $\kappa$.

Recall that a family $\cA$ of subsets of $\omega$ is almost disjoint if $A\cap B$ is finite for any distinct $A,B\in \cA$.
We write $A\cap B=^\ast\emptyset$ to say that $A,B$ are almost disjoint, and $A\sub^\ast B$ to denote that $A$ is almost contained in $B$.

To every almost disjoint family $\cA$ one can associate an Alexandroff-Urysohn  compactum $K_\cA$ of height 3. Such space may be simply defined as the Stone space of the algebra of subsets of $\omega$ generated by $\cA$ and all finite sets. In other words,
\[K_\cA=\omega\cup \{p_A:A\in\cA\}\cup\{\infty\},\]
where points in $\omega$ are isolated, basic open neighbourhoods of $p_A$ are of the form $\{p_A\}\cup (A\sm I)$ with $I\sub \omega$ finite, and $\infty$ is the one-point compactification of the locally compact space $\omega\cup \{p_A:A\in\cA\}$. Alexandroff-Urysohn compacta are of interest for us, since
the space $C(K_\cA)$ is a non-trivial twisted sum of $c_0$ and $c_0(\con)$
whenever $|\cA|=\con$, see \cite{CCMPS19} for details.

\subsection{Measures and Boolean algebras}
We shall discuss Boolean algebras and their Stone spaces using the classical Stone duality. Given an algebra $\fB$, its Stone space (of all ultrafilters on $\fB$) is denoted by $\ult(\fB)$. On the other hand, if $K$ is a zero-dimensional compact space then $\clop(K)$ is the algebra of clopen subsets of $K$. There is a natural isomorphism between $\fB$ and $\clop(\ult(\fB))$.

We write $M(\fB)$ for the space of all signed {\em finitely}  additive functions on an algebra $\fB$.
Likewise, for any $r > 0$, $M_r(\fB)$ denotes the family
of $\mu\in M(\fB)$ for which $\|\mu\|\le r$. Here, as usual, $\|\mu\|=|\mu|(1_\fB)$, where the variation $|\mu|$ is given by
\[|\mu|(A)=\sup_{B\in \fB, B\sub A} \{|\mu(B)| + |\mu(A\setminus B)|\}.\]
If $\mu\in M(\fB)$ and $\fB'\sub \fB$ is some subalgebra, then $\|\mu\|_{\fB'}$ will stand for the norm of the restriction $\mu|\fB'$.

The space $M(\fB)$ may be identified with $M(\ult(\fB))$ because every $\mu\in M(\fB)$ defines, via the Stone isomorphism, an additive function on $\clop(\ult(\fB))$ and such a function extends uniquely to a Radon measure. Note that, under such an identification, the $weak^\ast$ topology on $M(\fB)$ becomes the topology of pointwise convergence on elements of $\fB$.

Below we consider Boolean algebras contained in ${\mathcal P}(\omega)$, so let us introduce some terminology and notation for this particular case. We write $\la\cA\ra$ to denote the subalgebra generated by any $\cA\sub {\mathcal P}(\omega)$.

\begin{definition}
An algebra $\fB \sub \ps(\omega)$ is said to be \emph{trivial on a set} $S\sub\omega$ if for every $B\in\fB$ either $B\cap S$ or $B\sm S$ is finite.
\end{definition}

 The symbol $\ml(\cP(\omega))$ will denote the subspace of $M(\cP(\omega))$ consisting of the measures $\mu$ which are defined by an element $x\in \ell_1$, that is
 \[ \mu(B)=\sum_{n\in B} x(n),\]
  for $B\sub\omega.$
Below we often take some $\mu\in \ml(\cP(\omega))$ and consider the restriction of $\mu$ to some subalgebras of  $\cP(\omega)$.

We shall use  the well-known Rosenthal lemma in the following form (cf.\ \cite{DU77}, page 18).

\begin{theorem}\label{ros}
Let $(\mu_{n})_{n}$ be a uniformly bounded sequence of measures in $\ml(\ps(\omega))$. For every infinite set $S\sub\omega$ and $\eps>0$ there is an infinite subset $N\sub S$ such that $|\mu_k|(N\sm\{k\})<\eps$ whenever $k\in N$.
\end{theorem}

The following fact will also be used.

\begin{lemma}\label{ros2}
Let $(\mu_{n})_{n}$ be a uniformly bounded
sequence of measures in $\ml(\ps(\omega))$. Then for every infinite set $S\sub\omega$ there are infinite subsets $N,T\sub S$ such that
the subsequence $(\mu_k)_{k\in N}$ is norm convergent on the set $T$.
\end{lemma}

\begin{proof} There is no loss of generality if we suppose $S=\omega$. Since $(\mu_n)_n$ is uniformly bounded, it is a bounded subset of $\ell_1$, so by passing to a subsequence, we may assume that $(\mu_n)_n$ is $weak^\ast$ convergent to $\mu$. Now, replacing $(\mu_n)_n$ with $(\mu_n -\mu)_n$ if necessary, we can also assume that $\mu=0$.
\par We construct by induction  increasing sequences of natural numbers $(n_k)_k$, $(m_k)_k$ so that
	\[ \sum_{j=1}^{m_{k-1}} |\mu_n(j)| \le {2^{-k-1}} \mbox{ for all } n\geq n_k \quad\text{ and }\quad \sum_{j=m_k}^\infty |\mu_{n_k}(j)|\le
{2^{-k-1}}\]
Let $N=\{n_k: k\in\omega\}$ and $T=\{m_k:k\in\omega\}$. It is now straightforward to check that $(\mu_k)_{k\in N}$ is norm convergent to $0$ in $T$, since for each $k\in\omega$
$$\sum_{j\in T} |\mu_{n_k}(j)| \leq \sum_{j=1}^{k-1}|\mu_{n_k}(m_j)| + \sum_{j=k}^{\infty}|\mu_{n_k}(m_j)| < 2^{-k}$$
\end{proof}

\par Last, given a sequence of measures $(\mu_n)_n$ in $\ml(\ps(\omega))$ and two subalgebras $\fB \subseteq \fC \subseteq \ps(\omega)$, we will need to discuss up to what extent clustering of $(\mu_n)_n$ in $\fB$ differs from clustering in $\fC$. This is essentially the content of the following lemma.

\begin{lemma}\label{com}
Let $\fB$ be a subalgebra of $\ps(\omega)$, $D\subseteq \omega$ and let $\fC\sub \la \fB\cup\cP(D)\ra$.
Assume $(\mu_n)_n$ is a sequence in $\ml\big(\cP(\omega)\big)$.

\begin{enumerate}[(a)]
\item If $(\mu_n)_n$  is convergent on $\fB$ and $|\mu_n|(D)\le \eps$ for some $\eps>0$ and every $n$ then  any cluster points $\mu',\mu''$ of
$\{\mu_n:n\in\omega\}$ in $M(\fC)$ satisfy $\|\mu'-\mu''\|_{\fC}\le 6\eps$.
\item If $(\mu_n)_n$ is norm convergent on $D$ then
 any two cluster points $\mu',\mu''$ of
$\{\mu_n:n\in\omega\}$ in $M(\fC)$ satisfy
\[ \|\mu'-\mu''\|_{\fB}=  \|\mu'-\mu''\|_{\fC}.\]
\end{enumerate}
\end{lemma}

\begin{proof} Note that we can assume that $D\in\fC$.

With the assumptions of  $(a)$, we have $|\mu'-\mu''|(D)\le2\eps$. For any $B\in\fB$, $\mu'(B)=\mu''(B)$, and  so
\[|\mu'(B\sm D)-\mu''(B\sm D)|= | \mu'(B)-\mu'(B\cap D) -(   \mu''(B)-\mu''(B\cap D) )|=\]
\[=|\mu''(B\cap D) -\mu'(B\cap D)|\le 2\eps.\]
This gives $|\mu'-\mu''|(\omega\sm D)\le 4\eps$, and hence
\[|\mu'-\mu''|(\omega)=|\mu'-\mu''|(D)+|\mu'-\mu''|(\omega\sm D)\le 6\eps.\]

Now clause $(b)$ follows by similar calculations since we now have $|\mu'-\mu''|(D)=0$.
\end{proof}

\section{The construction --- An almost disjoint family}

Our first step on the way to twisted sums is the construction of an almost disjoint family $\mathcal A$ such that that the corresponding compact space $K_{\mathcal A}$ has a certain peculiar property.
It will be convenient to introduce the following technical notion.

\begin{definition}\label{fork}
Given a compact space $K$ and $c>0$, $\eps>0$, we say that three measures $d \delta_{x_i}+\nu_i\in M(K)$, $i=0,1,2$,  form
a $(c,\eps)$\emph{-fork} if $|d|\ge c$ and

\begin{enumerate}[(i)]
\item $x_0, x_1, x_2$ are different points of $K$;
\item $\|\nu_i-\nu_j\| \le \eps$ for $i,j=0,1,2$.
% \item  $|b_i|-d|< \varepsilon$ for $i=0,1,2$.
%	\item $\nu_i\{x_i\}=0$ for $i=0,1,2$;
\end{enumerate}
\end{definition}

Below, working with some algebra $\fB$,  we say that a subset of $M(\fB)$ contains a fork (with parameters given) if
such a fork can be found in the corresponding set of measures on $M(\ult(\fB))$.

The following lemma describes the inductive step of the construction used in the proof of Theorem \ref{adf}.

\begin{lemma}$[\mathfrak{p}=\con]$\label{step}
Suppose that $\fB$ is an algebra on $\omega$ containing all finite sets such that $|\fB|<\con$;
let $S\sub\omega$ be an infinite set on which  $\fB$ is trivial.

Suppose that a sequence $(\mu_k)_{k\in S}$  in the unit ball of $\ml(\cP(\omega))$ satisfies
 $|\mu_k(\{k\})|\ge c$ for every $k\in S$, where $c>0$ is a fixed constant.
Then for  every $\eps>0$  there are infinite disjoint sets $N,T\sub S$, and pairwise disjoint sets $A_0,A_1,A_2\sub N$ such that, if we let $\fB'=\la \fB\cup \{A_0,A_1,A_2\}\ra $,
 then

\begin{enumerate}[(i)]
\item the sequence $(\mu_k)_{k\in N}$ is
norm convergent on $T$;
\item the closure of the set $\{\mu_k: {k\in N}\}$ in $M(\fB')$ contains a $(c,\eps)$-fork.	
\end{enumerate}
\end{lemma}

\begin{proof}
Write  $\mu_k = b_k\delta_k + \alpha_k$, where $\alpha_k(\{k\})=0$. Using $\pe=\con$ and passing to subsequences,
we can assume that

\begin{enumerate}[---]
\item the sequence of measures $(\mu_k)_{k\in S}$ converges on $\fB$;
\item  the sequence $(b_k)_{k\in S}$ converges to some $b$.
\end{enumerate}

Note that in such a case the measures $\alpha_k$ also converge on $\fB$ (since $\fB$ is trivial on $S$,
the sequence $(\delta_k)_{k\in S}$ is convergent on $\fB$).
Let us assume that $b \ge c$; the `negative' case will be symmetric.

Fix $\varepsilon>0$;
by virtue of Lemma $\ref{ros2}$, we can find infinite disjoint subsets $N, T\sub S$ so that $(\mu_k)_{k\in N}$ is norm convergent on $T$.
We now apply Lemma \ref{ros} to the sequence $(\mu_k)_{k\in N}$ to obtain an infinite set $D\sub N$
such  that $|\alpha_k|(D)\le  \varepsilon/6$  for all $ k\in D$.

Split $D$ into three infinite sets, $D=A_0 \cup A_1 \cup A_2$ and set
$\fB'=\la \fB \cup \{A_0, A_1, A_2\}\ra$.
Then $(i)$ is granted; we shall find the required $(c,\eps)$-fork.

For $i=0,1,2$, let  $x_i \in \ult(\fB')$ be the only ultrafilter containing $A_i$ but no finite sets.
Take any cluster point $\nu_i$ of the subsequence $(\alpha_k)_{k\in A_i}$.
It is clear that the measures $b\delta_{x_i}+\nu_i$ lie in the closure of $\{\mu_k: {k\in N}\}$ in $M(\fB')$,
and we have $\|\nu_i -\nu_j \|_{\fB'}\le \eps$ by Lemma \ref{com}(a), as required.
\end{proof}

\begin{theorem}\label{adf}
$[\mathfrak{p}=\con]$
There exists an almost disjoint family $\cA \sub \cP(\omega)$ such  that the algebra $\fB$ generated by
$\cA$ and finite sets has the following property:

For every $0<c<1$  and every sequence
 $(\mu_n)_n$ in  the unit ball of  $\ml(\cP(\omega))$ satisfying $|\mu_n(\{n\})| \ge c$ for  $n\in\omega$,
  the  closure of $\{\mu_k: k\in \omega\}$ in $M(\fB)$ contains  a $(c,c/2)$-fork.
\end{theorem}

\begin{proof}
We fix an enumeration
\[ \left\{ \left((\mu_k^\xi)_k, c_\xi\right):\xi<\con \right\},\]
 of all  sequences and constants that satisfy the assumption of the theorem.
We define inductively almost disjoint sets $N_\xi$ (room to perform the step $\xi$),
and auxiliary infinite sets $T_\xi$ (making room for future steps of the construction).
Then for every $\xi$ we shall find pairwise disjoint infinite sets $A^\xi_0,A^\xi_1,A^\xi_2\sub N_\xi$
so that
\[ \cA=\{A_\xi^i:\xi<\con, i=0,1,2\},\]
 will be the required almost disjoint family.
%$\cA_\xi=\{\cA_\eta:\eta<\xi\}
%$\cA=\bigcup_{\xi<\con}\cA_\xi$

Writing $\fB_\xi$ for the algebra generated by $\{A^\eta_i:\eta<\xi, i=0,1,2\}$ and all finite sets in $\omega$,
we do the construction so that

\begin{enumerate}[(a)]
\item $N_\eta\cap T_\xi=^\ast\emptyset$ for $\eta<\xi$ and $N_\eta, T_\eta \sub^\ast T_\xi$ for $\eta\ge\xi$;
\item the sequence $(\mu^\xi_k)_{k\in N_\xi}$ converges on the algebra $\fB_\xi$;
\item  the sequence $(\mu^\xi_k)_{k\in N_\xi}$ converges in norm on the set $T_\xi$;
\item  the closure of $\{\mu_k^\xi:k\in\omega\}$ in $M(\fB_{\xi+1})$ contains a $(c_\xi, c_\xi/2)$-fork.
\end{enumerate}

Suppose that the construction has been done for $\eta<\xi<\con$. Using $\pe=\con$, we first find an infinite set $Y\sub\omega$  such that
$Y\sub^\ast T_\eta$ for every $\eta<\xi$. Since the algebra $\fB_\xi$ verifies  $|\fB_\xi|<\con$, again by $\pe=\con$, we find $S\sub Y$ such that
the sequence $(\mu^\xi_k)_{k\in S}$ converges on $\fB_\xi$.
Note that $\fB_\xi$ is trivial on $S$, so we can apply Lemma \ref{step} with $\varepsilon=c_\xi/2$ and get $N_\xi, T_\xi\sub S$ together with $A^\xi_i\sub N_\xi$, $i=0,1,2$, and
(a) --- (d) are fulfilled.

Now it suffices to note  that the fork obtained at step $\xi$ will not be destroyed in the future; this is a direct consequence of (c) and Lemma \ref{com}(b) since
the final algebra $\fB=\bigcup_{\xi<\con} \fB_\xi$ is contained in the algebra generated by $\fB_{\xi}$ and $\cP(T_\xi)$.
\end{proof}

For the purpose of the next section we need to augment Theorem \ref{adf} to the form given below.
Let us  write $[\omega]^\omega$ for the family of all infinite subsets of $\omega$. We say that
a family $\cN\sub [\omega]^\omega$ is {\em dense in }$[\omega]^\omega$ if every infinite $A\sub \omega$ contains
some $B\in\cN$. For the sake of example, given a sequence $(x_n)_n$ in a sequentially compact space, we can think of the family $\cN$ of infinite subsets $I\subseteq \omega$ such that $(x_n)_{n\in I}$ is convergent.

%every dense family $\cN\sub [\omega]^\omega$

\begin{theorem}\label{adf2}
$[\mathfrak{p}=\con]$
Let $\{ \cN_\alpha:\alpha<\con \}$ be a certain list of families that are dense in $[\omega]^\omega$.

There exists an almost disjoint family $\cA \sub \cP(\omega)$ such  that the algebra $\fB$ generated by
$\cA$ and finite sets has the following property:

For every $0<c<1$  and every sequence
 $(\mu_n)_n$ in  the unit ball of  $\ml(\cP(\omega))$ satisfying $|\mu_n(\{n\})| \ge c$ for  $n\in\omega$,
 {\bf and for every} $\alpha<\con$,
  the  closure of $\{\mu_k: k\in I\}$ in $M(\fB)$ contains  a $(c,c/2)$-fork, where $I\sub N$ for some $N\in\cN_\alpha$.
\end{theorem}

\begin{proof}
As before, we consider all the sequences of measures and constants in question
\[ \left\{ \left((\mu_k^\xi)_k, c_\xi\right):\xi<\con \right\},\]
together with the given list $\la \cN_\alpha:\alpha<\con\ra$ and do the induction over all pairs $(\alpha,\xi)$ (ordered in type $\con$).
The only modification that is needed: when we consider the sequence $(\mu^\xi_k)_k$ and the family $\cN_\alpha$ we apply
Lemma \ref{step} and, at the beginning of its proof, we can additionally assume that $S\in \cN_\alpha$.
\end{proof}

\section{The construction --- A twisted sum of $c_0$ and $C(K)$}
We now turn to the construction of a twisted sum $0\to c_0 \to X \to C(K)\to 0$, for an Eberlein compactum $K$ of weight $\con$ in which $X$  is not a $\cC$-space.
\par Let $K$ be a compact space. Every short exact sequence
\[ \begin{tikzcd}
0 \arrow[r] & c_0 \arrow[r] & X \arrow[r]  & C(K) \arrow[r] & 0
\end{tikzcd} \]
is equivalent to one arising from a countable discrete extension in the following way, see \cite{amp} for further details. Take the dual unit ball $M_1(K)$ and consider a countable discrete extension $L= M_1(K)\cup\; \omega$. The space $C(K)$ canonically embeds into $C(M_1(K))$. Now consider the following diagram:
\[ \begin{tikzcd}
0 \arrow[r] & c_0 \arrow[r] \arrow[d, equal] & C(L) \arrow[r]  & C\big(M_1(K)\big) \arrow[r] & 0 \\
0 \arrow[r] & c_0 \arrow[r]           & PB \arrow[r]  \arrow[u,hook]         & C(K) \arrow[r]    \arrow[u, hook, "\iota"]       & 0
\end{tikzcd} \]
The pull-back space $PB$ is precisely our twisted sum $X$. Explicitly,
\[ X=\{f\in C(L): \exists g\in C(K) : f(\mu)=\mu(g) \ \forall \mu\in M(K)\}.\]
In other words, $X$ is the space of all possible extensions of a function $g\in C(K)$  to a continuous function $\wt{g}\in C(L)$. From now on, given $g\in C(K)$ we will write $\wt{g}$ to say that we treat $g$ as a function on $M_1(K)$.
%\par We have a natural embedding $c_0$ onto the subspace of $X$ of functions vanishing on $M_1(K)$ and the quotient $X/c_0$ is $C(K)$.

\par From the diagram above it also follows that $X^\ast$ may be identified with $M(K) \oplus \ell_1$. This is because the dual of the second row
\[ \begin{tikzcd}
0 \arrow[r] & M(K) \arrow[r] & X^* \arrow[r]  & \ell_1 \arrow[r] & 0
\end{tikzcd} \]
must split \cite[p. 77]{DJT}.
An explicit isomorphism can be given using the existence of barycenters of measures:
if $\vf\in X^\ast$, we may think that $\vf$ is represented by some measure $\Lambda\in M(L)$. Write $\Lambda=\Lambda_1+\Lambda_2$ where $\Lambda_1$ is concentrated on $M_1(K)$ and
$\Lambda_2$ is concentrated on $\omega$. Then it is clear that $\Lambda_2\in \ell_1$; on the other hand, $\Lambda_1$ is a measure on the compact convex set $M_1(K)$ so it has a barycenter $\mu\in M_1(K)$, i.e.\  that $\Lambda_1(\wt{g})=\mu(g)$ for $g\in C(K)$.

 Consider now an Eberlein compact space $K$ of weight $\con$. Recall that in such a case, $M_1(K)$ is also Eberlein compact in its $weak^\ast$ topology, and in particular, it is a sequentially compact space which satisfies $|M_1(K)|=\con$. This fact can be derived from the classical Amir-Lindenstrauss theorem \cite[Thm. 2]{amir-lin}; see also  Negrepontis  \cite{Ne84}.

We shall now use the almost disjoint family $\cA$ from the previous section to construct a certain countable discrete extension $L$ of $M_1(K)$ so that the corresponding space $X$ cannot be a $\cC$-space. Our first step is to locate a copy of $\bA(\con)$ in $M_1(K)$. 

\begin{lemma} 
	Let $K$ be an Eberlein compactum of weight $\mathfrak c$. 
	Then there are $\sigma_\alpha\in M_1(K)$ such that the set
	$\Sigma=\{\sigma_\alpha, -\sigma_\alpha: \alpha<\con\} \cup \{0\}$
	is  a copy of $\bA(\con)$ inside $M_1(K)$.
\end{lemma}

\begin{proof} We can assume that $K$ is a weak compact set of $c_0(\con)$; 
	then for every $x\in K$  the set $\supp x=\{\xi<\con : x(\xi)\neq 0\}$ is countable.
	
	For $\alpha<\con$ we inductively define a pair of elements $x_\alpha, y_\alpha$ in $K$ 
	and a set $I_\alpha$ so that
	
	\begin{enumerate}[(i)]
		\item $I_\alpha=\bigcup_{\beta<\alpha} \left( \supp{x_\beta} \cup \supp{y_\beta}\right)$;
		\item $x_\alpha\neq y_\alpha$ but $x_\alpha(\xi)=y_\alpha(\xi)$ for  every $\xi\in I_\alpha$.
	\end{enumerate}
	
	To carry out the construction  just note that for  $\alpha<\con$ we have  $|I_\alpha|<\con$ so, 
	by  the assumption that $K$ has weight $\con$,
	there is  a pair of points $x_\alpha \neq y_\alpha$ in $K$ which agree on $I_\alpha$. 
	
	Now set
	\[ \sigma_\alpha=\frac12(\delta_{x_\alpha} - \delta_{y_\alpha}) \ , \ \alpha < \con;\]
	we will show that $\Sigma=\{\sigma_\alpha, -\sigma_\alpha: \alpha<\con\} \cup \{0\}$ is a copy of $\bA(\con)$ in $M_1(K)$. 
	
	It is easy to see that $\sigma_\alpha$ and $-\sigma_\alpha$ are isolated in $\Sigma$.
	Since the functions $f_S = \prod_{\xi\in S} \pi_\xi$, where $S\subset I$ is a finite set, 
	span a subalgebra of $C(K)$ which separates the points in $K$, it suffices to check that for every such $S$, 
	we have 
	\[(*)\quad \int_K f_S\;{\rm d}\sigma_\alpha \neq 0,\]
	for at most one $\alpha$.
	
	Indeed, if $S$ is not contained in $\bigcup_{\alpha<\con}  I_\alpha$  then clearly $\int_K f_S\;{\rm d}\sigma_\alpha = 0$
	for all $\alpha$;
	otherwise,  take $\alpha_0=\min\{\alpha<\con: S\sub I_\alpha\}$ and note
	that $\int_K f_S\;{\rm d}\sigma_\beta = 0$ for all $\beta>\alpha_0$ by $(ii)$
	and for all $\beta<\alpha_0$ by $(i)$.
\end{proof}

\begin{theorem}\label{main} $[\pe=\con]$. If $K$ is an Eberlein compactum of weight $\con$, then there is a non-trivial twisted sum
\[ 0 \longrightarrow c_0 \longrightarrow X \longrightarrow C(K) \longrightarrow 0,\]
where the Banach space $X$ is not a $\cC$-space.
\end{theorem}

\begin{proof}
%Let $\{p_\xi: \xi<\con\}\sub K$ be a discrete set such that its closure $\{p_\xi: \xi<\con\}\cup \{p\}$ is a copy of $\bA(\con)$ inside $K$.
%Set
%\[ \sigma_\xi=1/2\left(\delta_{p_{\xi}} - \delta_{p_{\xi+1}}\right)\in M_1(K),\]
%for $\xi<\con$. Note that $\Sigma=\{\sigma_\xi , -\sigma_\xi: \xi<\con\}\cup\{0\}$ is a copy of $\bA(\con)$ inside $M_1(K)$.

\par The hypothesis $|M(K)|=\con$ allows us to enumerate all the sequences in $M_1(K)$ as $\{( \lambda^\alpha_n)_n: \alpha<\con\}$.
For a given $\alpha$, let $\cN_\alpha\sub [\omega]^\omega$ be the family of those infinite $I\sub\omega$ for which
the subsequence $(\lambda^\alpha_n)_{n\in I}$ is $weak^\ast$ convergent. Note that every family $\cN_\alpha$ is dense in
 $[\omega]^\omega$ since $M_1(K)$ is sequentially compact in its $weak^\ast$ topology.

 We apply Theorem \ref{adf2} with $\la \cN_\alpha:\alpha<\con\ra$ defined above, and
 consider the almost disjoint family $\cA= \{A_\xi^i:\xi<\con, i=0,1,2\}$ constructed in  \ref{adf2}.
The points in $K_\cA'$ corresponding to $A\in\cA$ will be denoted by $x(A)$, i.e.\
$x(A)\in\ult (\fB)$ is the only ultrafilter on $\fB$ containing $A$ and no finite set.

We now define a countable discrete extension $L=M_1(K)\cup\omega$. For this purpose, let $\Sigma = \{\sigma_\xi, -\sigma_\xi, 0\}$ be a copy of $\bA(\con)$ as in the previous lemma, and also consider a (disjoint) union of $M_1(K)\cup\K_\cA$. Our space $L$ is the result of gluing $\Sigma$ with $K_\cA'$ as follows:
\begin{enumerate}[--]
\item $x(A_\xi^1)$ is identified with $\sigma_\xi$;
\item $x(A_\xi^2)$ is identified with $-\sigma_\xi$;
\item $x(A_\xi^0)$ is identified with $0\in M_1(K)$.
\end{enumerate}

Write $q: M_1(K)\cup  K_\cA\to M_1(K)\cup\omega$ for the corresponding quotient map.
 Let us show that the twisted sum $X$ of $c_0$ and $C(K)$ given by such a discrete extension of $M_1(K)$, as described at the beginning of this section,
 cannot be a $\cC$-space.
 Let $F$ be a $c$-norming subset of $B_{X^\ast}$ for some $c>0$. Then $F$  must contain a sequence $(\mu_n)_n$ such that
 $|\mu_n\{n\}|\ge c$ for every $n$. Since $X^*=M(K)\oplus \ell_1$, we can decompose $\mu_n = \mu_n^K + \mu_n^\omega$ into measures supported by
  $K$ and $\omega$, respectively.

 The main point is that there is $I\sub\omega$ such that simultaneously,
 \begin{enumerate}[(i)]
     \item the closure of $\{\mu_n^\omega: n\in I\}$ in $M(K_\cA)$ contains a $(c,c/2)$-fork, and
     \item the sequence $(\mu^K_n)_{n\in I}$ is $weak^\ast$ convergent in $M(K)$ to some $\mu$.
	\end{enumerate}

 Let the fork mentioned above consist of the measures
 \[ b \delta_{x_i}+\nu_i \in M(K_\cA) \mbox{  for  } i=0,1,2,\]
 where we write $x_i=x(A_\xi^i)$ for simplicity.
 Observing the way $K_\cA$ and $K$ are glued together we conclude that the closure of $\{\mu_n: n\in\omega\}$  in $M(L)$ contains the following triple
\[ %\mbox{T}\quad
b\sigma_\xi +\nu_1', \quad -b\sigma_\xi+\nu_2',\quad  \nu_0';\]
here $\nu_i'=q[\nu_i]$ is the image of $\nu_i$ under the quotient map $q$. Recall that
$\|\nu_i-\nu_j\|\le c/2$ for $i,j=0,1,2$;
hence $\|\nu_i'-\nu_j'\|\le c/2$ for $i,j=0,1,2$, and therefore
\[ \left\|1/2\big((b\sigma_\xi +\nu_1' + \mu)+(  -b\sigma_\xi+\nu_2'+\mu)\big)-(\nu_0'+\mu)\right\|=\|1/2(\nu_1'+\nu_2')-\nu_0'\|\le c/2.\]
We have found a forbidden $c$-triple (see Definition \ref{p:3}).

It follows that  $F$ is not a free set;
by virtue of Lemma \ref{p:1}, $X$ is not a $\cC$-space, and the proof is complete.
\end{proof}

\section{Properties of twisted sums of $C(K)$-spaces}

To be a $\cC$-space is not a $3$-space property: indeed, there exists twisted sums of $c_0$ and $C(K)$ that are not $\cC$-spaces when either $K$ is an Eberlein compactum of weight $\con$ (under $\mathfrak p= \mathfrak c$, Theorem \ref{main}) or when $K= \beta \omega\setminus \omega$ by \cite{ben}. It is an open problem \cite[Remark p.67]{kaltCC} whether $\beta \omega$ can also be added to this list. When $c_0$ is replaced by a more complex $\cC$-space then more examples are known: in \cite{ccky}, there are presented examples of exact sequences $0 \longrightarrow C(\omega^\omega) \longrightarrow Z \longrightarrow c_0 \longrightarrow 0$ and $0 \longrightarrow C[0,1] \longrightarrow Z \longrightarrow X \longrightarrow 0$ with strictly singular quotient map for every $X$ not containing $\ell_1$. That prevents $Z$ from having Pe\l czy\'nski's property $(V)$ and therefore from being a $\cC$-space.

Let us focus on twisted sum spaces $Z$ of $c_0$ and $c_0(\aleph)$. In \cite[\S4]{v} it was shown that every twisted sum of $c_0$ and a Banach space with Pe\l czy\'nski's property $(V)$ also has property $(V)$. We improve that result by showing:

\begin{theorem} Every twisted sum of $c_0$ and a Lindenstrauss space is isomorphic to a Lindenstrauss space. Consequently, it has property $(V)$.
\end{theorem} \begin{proof} The key point is that every twisted sum of $c_0$ and a Banach space $X$, not necessarily a $\cC$-space, can be obtained as the lower row of the following diagram \cite{amp}:
		\[ \begin{tikzcd}
	0 \arrow[r] & c_0 \arrow[r] \arrow[d, equal] & C(B_{X^*}\cup \omega) \arrow[r]  & C(B_{X^*}) \arrow[r] & 0 \\
	0 \arrow[r] & c_0 \arrow[r]           & Z \arrow[r]  \arrow[u,hook]         & X \arrow[r]    \arrow[u, hook, "\iota"]      & 0
	\end{tikzcd} \]
	
\noindent In this context, the adjoint operator $\iota^*$ maps every measure $\mu \in M(B_{X^*})$ to its barycenter $\Bar \mu\in X^*$. Now let us look at the exact sequence:
	\[ 0 \longrightarrow Z^\perp \longrightarrow M(B_{X^*}\cup \omega) \longrightarrow Z^* \longrightarrow 0\]
It is clear that $Z^\perp = \{\mu\in M(B_{X^*}\cup \omega): \mu|_{\omega}=0, \Bar\mu|_{B_{X^*}}=0 \}$ and we only need to compute the norm of an element $[\mu]\in Z^*$:
	$$\|[\mu]\| = \inf_{\nu \in Z^\perp} \|\mu-\nu\| = \|\mu|_\omega\| + \inf_{\nu \in Z^\perp} \|(\mu-\nu)|_{B_{X^*}}\|$$
where $\|(\mu-\nu)|_{B_{X^*}}\|\geq\|\Bar(\mu|_{B_{X^*}})\|$ for any $\nu\in Z^\perp$.
Note that every element $x\in B_X$ defines the Dirac measure $\delta_x\in M(B_{X^*})$, and this assignment can be extended homogeneously to all $X$ by letting $Tx = \|x\|\cdot \delta_{x/\|x\|}$.
Therefore, the choice $\nu=\mu|_{B_{X^*}}-T(\Bar \mu)$ attains the minimum, which is $\|{\Bar(\mu|_{B_{X^*}}})\|$.
This yields that $Z^*$ is isometrically isomorphic to $\ell_1 \oplus_1 X^*$, which is an $\mathcal L_1$-space provided $X$ is Lindenstrauss.
\end{proof}

Thus, the question of how close is a twisted sum of $c_0$ and $c_0(\aleph)$ from a $\cC$-space stands. The standard way of obtaining such twisted sums is to form the Nakamura-Kakutani space $C(K_{\cA})$ obtained from an almost disjoint family $\cA$ of subsets of $\omega$ with cardinal at most $\mathfrak c$.
Of course, by our main theorem, there is a twisted sum $X(\con)$ which is not a $\cC$-space, but it does share properties with $\cC$-spaces, since it is isomorphic to a Lindenstrauss space. 
All spaces $C(K_{\mathcal A})$ are subspaces of $\ell_\infty$, and this suggests whether $X(\con)$ is a subspace of $\ell_\infty$ as well. A hint that points us in such direction is: appealing to either Sobczyk's theorem or property $(V)$, we obtain that the quotient map $C(K_{\mathcal A})\longrightarrow c_0(|\mathcal A|)$ is an isomorphism on some copy of $c_0$; that is, it is not $c_0$-singular. However, it follows from being a subspace of $\ell_\infty$ that it is $c_0(\aleph_1)$-singular (i.e., the restrictions of the quotient map to any copy of $c_0(\aleph_1)$ cannot be an isomorphism). One has:

\begin{proposition} $[\mathfrak p = \con]$ The twisted sum $0 \longrightarrow c_0 \longrightarrow X(\con) \longrightarrow c_0(\con) \longrightarrow 0$ 	is $c_0(\aleph_1)$-singular.
\end{proposition}

\begin{proof}
Assume this is not the case; that is, there is an into operator $j: c_0(I) \hookrightarrow c_0(\con)$ with $|I|=\aleph_1$ such that admits a lifting $E:c_0(\aleph_1) \to X(\con)$, as shown in the diagram.
	\[ \begin{tikzcd}
		0 \arrow[r] & c_0 \arrow[r] \arrow[d, equal] & X(\con) \arrow[r, "q"]  & c_0(\con) \arrow[r] & 0 \\
		0 \arrow[r] & c_0 \arrow[r]           & PB \arrow[r]  \arrow[u,hook]         & c_0(\aleph_1) \arrow[r]    \arrow[u, hook, "j"]  \arrow[ul, "E"]      & 0
	\end{tikzcd} \]

\par For every $\xi\in I$, let $e_\xi$ be the characteristic function of the singleton $\{\xi\}$, and pick $f_\xi\in X(\con)$ such that $q(f_\xi)=e_\xi$. Note that $f_\xi(\sigma_\xi)=\frac12$, hence, by density of $\omega$ in the countable discrete extension $M_1(\bA(\con)) \cup \omega$, there is $n_\xi\in \omega$ such that $f_\xi(n_\xi)\geq \frac14$.
Now, the cardinality of $I$ yields the existence of an uncountable subset $J\subset I$ and an element $n\in \omega$ such that $n=n_{\xi}$ for every $\xi\in J$; that is, $n$ has been chosen uncountably many times.
Then, for every finite collection $\xi_1, ..., x_m\in J$, we obtain
	$$\left\|\sum_{k=1}^m e_{\xi_k}\right\|=1 \quad , \quad
	  \left\|\sum_{k=1}^m f_{\xi_k}\right\| \geq \frac{m}{4}$$
thus contradicting continuity of $E$.
\end{proof}

Finally, it was observed in \cite[Proposition 4.2]{CCMPS19} that under {\sf MA}($\aleph$) and $\aleph <\mathfrak c$ every twisted sum of $c_0$ and $c_0(\aleph)$ is a quotient of a $\cC$-space. It therefore makes sense to ask whether $X(\con)$ is a quotient of a $\cC$-space. The reverse way can also be followed: we can ask if it is true that under {\sf MA}($\aleph$) and $\aleph <\mathfrak c$ every twisted sum of $c_0$ and $c_0(\aleph)$ is a $\cC$-space.

\section*{Acknowledgements}
Thanks are due to an attentive referee, who spotted a mistake in an earlier version of the main theorem. The authors would also like to thank Jesús M. F. Castillo for his valuable suggestions and remarks about the content of the paper.

\section*{Conflict of interest}
The authors declare that they have no conflict of interest.

\section*{Data availability}
Data sharing not applicable to this article as no datasets were generated or analysed during the current study.


\begin{thebibliography}{99}
\bibitem{amir-lin} D. Amir and J. Lindenstrauss,
{\em The structure of weakly compact sets in Banach spaces}, Ann. Math. 88 (1968), 35--46.

\bibitem{2132} A. Avil\'es, F. Cabello, J. M. F. Castillo, M. Gonz\'alez, Y. Moreno, \emph{ Separable injective Banach spaces}  Lecture Notes in Mathematics 2132 (2016) Springer-Verlag.



\bibitem{amp} A. Avil\'es, W. Marciszewski, G. Plebanek, \emph{Twisting $c_0$ around nonseparable Banach spaces}, arxiv 1902.07783v1. Advances in Mathematics 369 (in press)


\bibitem{ben} Y. Benyamini, \emph{An $M$-space which is not isomorphic to a $C(K)$ space,} Israel J.
Math, 28 (1977) 98--102.


\bibitem{hmbst}F.\ Cabello S\'anchez, J.M.F.\ Castillo, \emph{Homological Methods in Banach Space Theory}, Cambrdge Studies in Advances Mathematics, Cambridge Univ. Press, to appear.
	
\bibitem{CCMPS19}F.\ Cabello S\'anchez, J.M.F.\ Castillo, W.\ Marciszewski, G.\ Plebanek, A.\ Salguero Alarc\'on,
	{\em Sailing over three problems of Koszmider}, J. Funct. Anal. 279 (2020), no. 4


\bibitem{ccky} F. Cabello S\'{a}nchez, J.M.F. Castillo, N.J. Kalton and
D.T. Yost, \emph{Twisted sums with $C(K)$ spaces,} Trans. Amer.
Math. Soc. 355 (2003) 4523-4541.
	
\bibitem{ccy}
F. Cabello S\'anchez, J.M.F. Castillo, and D. Yost,
\emph{Sobczyk's theorems from A to B}, Extracta Math. 15 (2000)
391--420.

\bibitem{Ca16}
J.M.F.\ Castillo, {\em Nonseparable $C(K)$-spaces can be twisted when K is a finite height compact},  Topology Appl.\ 198 (2016), 107--116.

\bibitem{3space}
J.M.F.\ Castillo, M. Gonz\'alez, {\em Three space Problems in Banach Space Theory},  Lecture Notes in Mathematics 1667 (1997) Springer-Verlag.

%\bibitem{cms}  J.M.F. Castillo, D. Morales, J. Suárez de la Fuente; \emph{Derivation of vector-valued complex interpolation scales}, J. Math. Anal. Appl. 468 (2018) 461-472


\bibitem{efe}  J.M.F. Castillo, P.L. Papini, \emph{Epheastus account on Trojanski's polyhedral war}, Extracta Math. Vol. 29 (2014) 35-51.


\bibitem{v} J.M.F. Castillo, M. Simoes, \emph{Property $(V)$ still fails to be a $3$-space property},  Extracta Math. 27 (2012) 5-11.

\bibitem{Co18} C.\ Correa, {\em Nontrivial twisted sums for finite height spaces under Martin's Axiom}, Fund. Math., 248 (2020) 195-204.

\bibitem{CT16}
C.\ Correa and D.V.\ Tausk, {\em  Nontrivial twisted sums of $c_0$ and $C(K)$},  J.\ Funct.\ Anal.\ 270 (2016),  842--853.

\bibitem{DJT} D.\ Diestel, H. Jarchow and A. Tonge, {\em Absolutely summing operators},  Math.\ Surveys 15, AMS (1977).

\bibitem{DU77} D.\ Diestel, J.J\ Uhl, {\em Vector measures},  Math.\ Surveys 15, AMS (1977).

\bibitem{Fr84}
D.H.\ Fremlin,  {\em Consequences of Martin's axiom},
Cambridge Tracts in Mathematics  84, Cambridge University Press, Cambridge (1984).

\bibitem{kaltCC} N.J. Kalton, \emph{Lipschitz and uniform embeddings into $\ell_\infty$}, Fundamenta Math. 212 (2011) 53--69.

\bibitem{maple}  W.\ Marciszewski and G.\ Plebanek, \emph{Extension operators and twisted sums of $c_0$ and
	$C(K)$ spaces}, J. Func. Anal. 274 (2018) 1491--1529.
	
\bibitem{Ne84}
S.\ Negrepontis, {\em Banach spaces and topology}; in: Handbook of set-theoretic topology, 1045--1142,
North-Holland, Amsterdam (1984).

\bibitem{Ro03}
H.\ Rosenthal, {\em  The Banach spaces $C(K)$},
Handbook of the geometry of Banach spaces, Vol. 2, 1547--1602, North-Holland, Amsterdam, 2003.

\end{thebibliography}
\end{document}